\documentclass[11pt,a4paper]{amsart}
\usepackage[latin1]{inputenc}
\usepackage{amsxtra}
\usepackage{amssymb}
\usepackage{amsthm}
\usepackage{euler}
\usepackage{graphicx}
\input{xy}
\xyoption{all}
\usepackage[arrow,curve,matrix]{xy}
\usepackage[stable]{footmisc}

\evensidemargin=\oddsidemargin
\makeatletter
\def\cleardoublepage{\clearpage\if@twoside \ifodd\c@page\else
  \hbox{}
  \thispagestyle{empty}
  \newpage
  \if@twocolumn\hbox{}\newpage\fi\fi\fi}
\makeatother
\usepackage{fancyhdr}
\headwidth=\textwidth

\theoremstyle{plain} 
\newtheorem{mainresult}{Theorem}
\newtheorem{teo}[subsection]{Theorem}
\newtheorem{corollary}[subsection]{Corollary}
\newtheorem{proposition}[subsection]{Proposition}
\newtheorem{lemma}[subsection]{Lemma}

\theoremstyle{remark}
\newtheorem*{remark}{Remark}
\newtheorem*{notation}{Notation}

\theoremstyle{definition}
\newtheorem{map}[subsection]{Map}
\newtheorem{definition}[subsection]{Definition}
\newtheorem*{example}{\tt{Example}}
\newcommand{\circulito}{\underset{i}{\circ}}
\newcommand{\tensork}{\underset{{\tiny\hbox{$K$}}}{\otimes}}
\newcommand{\W}{{\bf{W} }}
\title{The Lie structure on the Hochschild cohomology of a modular group algebra.}
\author{Selene S{\'a}nchez-Flores}
\address{Universität Stuttgart, Institut für Algebra und Zahlentheorie, Pfaffenwaldring 57, D-70569, Stuttgart, Germany}
\email{sanchez@mathematik.uni-stuttgart.de} 
\date{March 2011}
\begin{document}
\maketitle
\let\thefootnote\relax\footnotetext{Research supported by Conacyt and in part by Conicet project MathAmsud Nocomalret.}
\begin{abstract}
We prove that the Gerstenhaber bracket 
on the Hochschild cohomology of the group algebra of a cyclic group 
over a field of positive characteristic is not trivial. In this case,  
we relate the Lie algebra structure on the odd degrees of the Hochschild cohomology with a Witt-type algebra. 
\end{abstract}

\section*{Introduction.}
The Hochschild cohomology of an associative unital algebra is endowed with 
a Gerstenhaber algebra structure. For finite dimensional algebras, 
the graded commutative structure of the Hochschild cohomology has been widely studied. 
In this paper, we are interested in the graded Lie algebra structure of the Hochschild cohomology 
given by the Gerstenhaber bracket. 
We will consider the group algebra of the cyclic group over a field of positive characteristic. 
In \cite{sanchez}, we proved that the Hochschild cohomology groups are zero in degrees $\geq 2$ 
if the first Hochschild cohomology group is semisimple, when the algebra is monomial over a field of characteristic zero. 
However, the above statement is no longer true if the characteristic is positive. 
The main objective for computing the Hochschild cohomology groups for monomial algebras whose first Hochschild cohomology group is semisimple as Lie algebra, 
was to study the Lie module structure of the Hochschild cohomology groups of higher degrees. Since we obtained that they are trivial in that case we consider now the positive characteristic case. For instance, if one takes the group algebra of the cyclic group of order equal to the characteristic of the field then the Hochschild cohomology groups are non zero and the first Hochschild cohomology group is isomorphic to, as Lie algebra, the Witt algebra, which is simple in case the characteristic of the field is not 2. The next step is to examine the Gerstenhaber bracket and in particular, we would like to determine whether it is trivial or not. We would like to point out that very few non trivial examples are known for the computation of the bracket and therefore for the understanding of the graded Lie algebra structure on the Hochschild cohomology. The principal complication for providing examples is that this bracket is defined using the Hochschild complex while smaller complexes are used to compute Hochschild cohomology. 
In this article, we prove that the Gerstenhaber bracket is not trivial in the case of the group algebra of cyclic group over a field of positive characteristic. 
Moreover, we give a formula for the bracket that allows us to determine the Lie module structure on the Hochschild cohomology groups. 
We are also able to relate the Lie algebra structure on the odd Hochschild cohomology with a a Witt type algebra.

Let $p$ be an odd prime, let $K$ be a field of characteristic $p$.  
Let $G=<g>$ be the cyclic group of order $|G|$, where $p$ divides $|G|$. 
Set $A:=KG$ to be the group algebra of $G$. 
Denote by $HH^*(A)$ the {\it{Hochschild cohomology of $A$}}, i.e. 
\[
HH^*(A):=Ext^*_{A^e}(A,A)
\]
where $A^e=A\otimes A^{op}$ is the {\it{enveloping algebra}} of $A$. 
The Lie algebra of derivations of $A$, which is isomorphic to $HH^1(A)$ since there are no inner derivations, 
is called the {\it{Witt algebra}} 
(see \cite{jacobson}, Chapter V, Ex. 21).  
It is generated by the following derivations. 
For $i=0,\dots,p-1$,  
let $D_{i}:A \rightarrow A$ 
be the derivation given by 
$D_i(g)=g^{i+1}$.
The commutator bracket in terms of the generators is then, 
\[
[\, D_i \, , \, D_j \,]=(j-i)D_{i+j}.
\]
In addition, the graded algebra structure on 
$HH^*(A)$ given by the cup product 
has been also described in \cite{holm} and 
in a more general situation in \cite{cibilssolotar} and \cite{suarez}. 
In \cite{cibilssolotar}, it is shown that  
\[
HH^*(A) \cong H^*(G,K) \otimes A. 
\]
as a graded algebra, 
where $H^*(G,K)$ is the group cohomology algebra with coefficients in the trivial $A$-module $K$. 
In fact, explicit isomorphisms are given in \cite{cibilssolotar} at the cochain level 
in a more general case: for any abelian finite group and any field. 
These isomorphisms enable us to transfer the Gerstenhaber bracket from the Hochschild cohomology 
to  $H^*(G,K) \otimes A$. 

The purpose of this paper is to study the Lie algebra structure 
on the Hochschild cohomology of $A$ via $H^*(G,K) \otimes A$. 
In this setting, 
we are able to relate this structure to a Witt-type graded Lie algebra as we explain next. 
It is well known that $H^n(G,K) \cong K$ for $n \geq 0$ 
since the characteristic of the field $K$ divides the order of the cyclic group $G$. 
This computation can be obtained by using the minimal resolution of $K$ as a trivial $A$-module.   
We define in this paper explicit isomorphisms 
$$s_n:H^n(G,K) \rightarrow K \quad \text{ and } \quad q_n:K \rightarrow H^n(G,K)$$
that allow us to give an explicit basis in $H^*(G,K)$ 
and to show that the Gerstenhaber bracket in $H^*(G,K) \otimes A$, and hence in $HH^*(A)$, is not trivial. 
Let us introduce some notation in order to state the main results. 

Set the following total order relation on the set $G$ as follows:
\[
1 < g < g^2 < \cdots < g^{|G|-1}
\]
For any element $x$ in $G$ define 
$$
Q(x):=\{ \,  y \in G \text{ such that } x\, y \,< \,y \, \}.
$$
Given $n\geq2$ and $\underline{x}=[\,x_1 \, | \, \cdots \, | \, x_i \, | \, x_{i+1} \, | \,\cdots\,| \, x_n \, ]$ in $G^{\times n}$, 
we will say that $\underline{x}$ {\it{ satisfies the condition}} $C(n)$ if and only if :

\begin{itemize}
\item $x_i \in Q(x_{i+1}) \; \forall \; i$  \underline{odd}, $i \geq 1$ when {\it{$n$ is even}} or  
\item $x_i \in Q(x_{i+1}) \; \forall \; i$  \underline{even} $i \geq 1$ when {\it{$n$ is odd}}. 
\end{itemize}

In the next result we give an explicit basis  of $H^*(G,K)$. 
\begin{mainresult}
Let $p$ be an odd prime. 
Denote $A=KG$ where the characteristic of $K$ is $p$ 
and $G$ is the cyclic group of order $|G|$ generated by $g$.  
Assume that $p$ divides $|G|$. 
Set $\beta^0=1_K$. 
Let $\beta^1$ be the derivation in $H^1(G,K)$ given by $\beta^1(g^i)=i$. 
For $n > 1$, let $\beta^n$ be the element in $H^n(G,K)$ represented by the map 
$\beta^n: G^{\times n} \rightarrow K$ defined as follows 
$$
\beta^{n} (\underline{x})= \begin{cases} 1 & \text{if $n$ is even and \underline{x} satisfies $C(n)$} \\
\beta^1(x_1) &  \text{if $n$ is odd and \underline{x} satisfies $C(n)$}\\
0 & \text{otherwise} \end{cases}
$$
where $C(n)$ is the condition given above. 
Then the set $\{\beta^n\}_{n \geq 0}$ is a basis of the $k$-vector space $H^*(G,K)$. 
\end{mainresult}

Moreover, we explicit the Gerstenhaber bracket in terms of this basis. 

\begin{mainresult}
Let $p$ be an odd prime. 
Denote $A=KG$ where the characteristic of $K$ is $p$ 
and $G$ is the cyclic group of order $|G|$.  
Assume that $p$ divides $|G|$. 
Let
$$\varphi: H^*(G,K) \times G \rightarrow K$$ 
be the map which is linear in the first variable and additive in the second variable such that 
$$\varphi(\beta^n,x)=\begin{cases} \beta^1(x) & \text{ if $n$ is odd} \\ 0 & \text{ otherwise} \end{cases}$$ 
where $\beta^n$ is an element of the basis of $H^*(G,K)$ described above and $x$ is in $G$. 
Then the Gerstenhaber bracket on $H^*(G,K) \otimes A$ is given by 
\[
[\, \beta^n \otimes x \,, \, \beta^m \otimes y ]=  (\varphi(\beta^n,y) - \varphi(\beta^m,x)) \beta^{n+m-1} \otimes xy\] 
where $x$ and $y$ are in $G$.
\end{mainresult}

This paper is organized as follows. 
In the first section, 
we transfer the Gerstenhaber bracket 
from $HH^{*+1}(A)$ to 
$HH^{*+1}(G,K) \otimes A$ where $A$ 
is the group algebra of any finite abelian group, 
using the explicit isomorphism given in \cite{cibilssolotar}. 
Then in section 2, 
we define the maps between 
the bar resolution and the minimal projective resolution 
of the trivial $A$-module $K$ for $G$ cyclic and we give explicit isomorphisms  
between $H^n(G,K)$ and $K$. 
In sections 3 and 4, 
we prove that the maps introduced in section 2 
are comparison maps between the bar resolution and the minimal projective resolution. 
In section 5, we compute the Gerstenhaber bracket 
and study the Lie structure on $HH^{*+1}(A)$ where 
$A=KG$ with $G$ the cyclic group of order $|G|$ and $p$ divides $|G|$.

\thanks{{\bf{Acknowledgment.}} I would like to thank Andrea Solotar for her suggestions and improvements to this article. 
\section{Gerstenhaber bracket.} 
Given a field $K$ and an associative $K$-algebra $A$,  
the Gerstenhaber bracket, 
\[
[\, - \, , \, - \,]_G: HH^n(A) \times HH^m(A) \longrightarrow HH^{m+n-1}(A) \, ,
\]
is defined 
on the Hochschild cohomology groups using the Hochschild complex. 
We begin this section by recalling the definition of this bracket. 
Let   
\[
C^n(A):= Hom_{K}(A^{\otimes^n_K},A),  
\]
for $n\geq 0$ be the space of {\it{Hochschild $n$-cochains}}. 
In \cite{gersten}, Gerstenhaber defined a right 
pre-Lie system $\{ C^n(A), \{ \circ_i \}_{i=1}^n\}$ 
where elements of $C^n(A)$ are declared to have degree $n-1$. 
The $K$-bilinear map
\[
\circ_i: C^n(A) \times C^m(A) \longrightarrow C^{n+m-1}(A)
\]
is given by the following formula:
\[
\scalebox{0.95}{$
f^n \circ_i g^m (x_1 \otimes \cdots \otimes x_{n+m-1}):=
f^n(x_1 \otimes \cdots \otimes g^m(x_i \otimes \cdots \otimes x_{i+m-1}) 
\otimes \cdots \otimes x_{n+m-1})$}
\]
where $m\geq 1$, $f^n$ is in $C^n(A)$ and $g^m$ is in $C^m(A)$. 
Suppose now that $m=0$, so let $f^n$ be in $C^n(A)$ and $a$ in $A$. 
For $n \geq1$ and $i=1,\dots,n$ define 
\[
\scalebox{1}{$
f^n \circ_i a \,(x_1 \otimes \cdots \otimes x_{n-1}):=
f^n(x_1 \otimes \cdots \otimes \underbrace{a}_{i-th} 
\otimes \cdots \otimes x_{n-1})$}.
\]
Gerstenhaber proved that this pre-Lie system 
induces a graded pre-Lie algebra structure on $C^{*+1}(A)$
by defining an operation $\circ$ as follows:
\[
f^n \circ g^m:= \sum_{i=1}^n (-1)^{(i-1)(m-1)} f^n \circ_i g^m.
\]
In case $n=0$, set $a \circ g^m:=0$ for all $a$ in $A$. 
Finally, $C^{*+1}(A)$ becomes a 
graded Lie algebra by defining the bracket 
as the graded commutator of $\circ$. So,  
\[
[f^n \, , \, g^m]_G:=f^n \circ g^m - (-1)^{(n-1)(m-1)} g^m \circ f^n.
\]  
Moreover, Gerstenhaber proved that if $\delta$ 
is the differential in the Hochschild complex, 
\[
\delta[f^m \, , \, g^n]_G= 
[f^m \, , \, \delta g^n]_G + (-1)^{n-1}[\delta f^m \, , \, g^n]_G.
\]
The above formula is obtained from the graded Jacobi identity and the fact that 
$$\delta f^n=(-1)^{n-1}[\, \mu \, , \, f^n \,]_G$$ where $\mu$ is the multiplication of $A$. 
Then 
the $K$-bilinear map
\[
[\, - \, , \, - \,]_G: HH^m(A) \times HH^n(A) \longrightarrow HH^{n+m-1}(A) 
\]
is well defined.
Therefore, 
$HH^{*+1}(A)$ endowed with the induced Gerstenhaber 
bracket is also a graded Lie algebra.

\subsection*{Gerstenhaber bracket on $H^{*+1}(G,K) \otimes A$.}
From now on, let $G$ be any abelian finite group and $A=KG$. 
In \cite{cibilssolotar}, 
explicit morphisms are given for each $n$ from $C^n(A)$ 
to $Map(G^{\times n},K) \otimes A$,  
in order to show that the graded algebra 
$HH^*(A)$ with the cup product is isomorphic 
to the tensor product algebra of the group cohomology algebra $H^*(G,K)$ and $A$. 
In this paragraph we will transfer the Gerstenhaber bracket from 
$HH^{*+1}(A)$ to $H^{*+1}(G,K) \otimes A$. 
To do so, we will use the explicit isomorphisms given in \cite{cibilssolotar}.

First, recall that the group cohomology $H^*(G,K)$ can be computed from the usual cochain complex 
$$
\bold{C}:=0 \rightarrow K \stackrel{d_0}{\longrightarrow} Map(G,K) \stackrel{d_1}{\longrightarrow} \cdots 
Map(G^{\times n},K) \stackrel{d_n }{\longrightarrow} Map(G^{\times n+1},K)
\cdots
$$
where $d_0=0$ and for $n \geq 1$
$$
\begin{array}{lcl}
d_n(f)(x_1 | \cdots | x_{n+1}) &=& f(x_2 | \cdots | x_{n+1}) \\
\, &+&
\displaystyle{\sum_{i=1}^n (-1)^i f(x_1 | \cdots | x_i x_{i+1} | \cdots | x_{n+1} )}  \\
\, &+& (-1)^{n+1} f(x_1 | \cdots | x_n).
\end{array}
$$
Denote $\bold{C'}:=\bold{C} \tensork A$ and $d'=d \tensork id$. 
In \cite{cibilssolotar}, the authors defined a cochain complex map 
from the Hochschild complex to ${\bold{C'}}$. In order to define it, let us introduce some notation.

\begin{notation}
Since $G$ provides a basis to the $K$-vector space $A$, given $a\in A$, 
there exists a unique $\lambda_x(a) \in K$ for all $x \in G$ such that 
$$a=\sum_{x \in G} \lambda_x(a) \, x.$$
Given $x \in G$ and $f \in Hom_K(A^{\otimes n},A)$, 
we define $\varepsilon^n(f,x)$ to be the map 
$$\varepsilon^n(f,x):G^{\times n} \rightarrow K$$ 
such that 
$$\varepsilon^n(f,x)(\, x_1\, | \, \cdots \, | \, x_{n} \, )= 
\lambda_{x_1\cdots x_nx} (f(x_1 \otimes \cdots \otimes x_n)).$$
\end{notation}

\begin{map}\label{phi} Set $\phi$ to be the family of maps $(\phi_n)$ where each 
\[
\begin{array}{rclc}
\phi_n: &Hom_K(A^{\otimes n},A) &\rightarrow &Map (G^{\times n},K) \otimes A \\
\end{array}
\]
is given by 
\[
\phi_n(f)= \displaystyle {\sum_{x \in G} }\; \varepsilon^n(f,x) \otimes x \,.
\]
Notice that $\phi_0=id_{A}$.
\end{map}
We will give an inverse map of $\phi$. To do so we need the following notation. 

\begin{notation}
For any map $\alpha$ in $Map(G^{\times n},K)$ and any element $x$ in $G$, 
we define $f^n_{(\alpha,x)}$ to be the linear map 
$$f^n_{(\alpha,x)}: A^{\otimes n} \rightarrow A$$ 
such that 
$$f^n_{(\alpha,x)}(x_1 \otimes \cdots \otimes x_n)=
\alpha(\, x_1\, | \, \cdots \, | \, x_{n} \,)\, x_1\cdots x_{n}x $$
where $x_i \in G$. 
\end{notation}

\begin{map}\label{psi}
We denote by $\psi$ the inverse map of $\phi$ given by the family of maps $(\psi_n)$ where each 
\[
\begin{array}{rclc}
\psi_n: &Map (G^{\times n},K) \otimes A &\rightarrow &Hom_K(A^{\otimes n},A) 
\end{array}
\]
is given by 
\[
\psi_n(\alpha \otimes x)=f^n_{(\alpha,x)}
\]
\end{map}

\begin{remark}
Let $f \in Hom_K(A^{\otimes n}, A)$. A straightforward computation shows that 
$$
\displaystyle {\sum_{x \in G} }\; f^n_{(\varepsilon^n(f,x),  x)} = f.
$$
Moreover, let $\alpha \otimes x$ be in $Map(G^{\times n},K) \otimes A$ then 
$$
\varepsilon^n(f^n_{(\alpha \otimes x)}, y)= 
\begin{cases}
\alpha & \text{ if } x=y \\
0 & \text{ otherwise. }
\end{cases}
$$
Therefore, the composition   
$\psi_n \phi_n$ is the identity map of $Map(G^{\times n},K) \otimes A$, while  
 $\phi_n \psi_n$ is the identity map of $C^n(A)$. 
This implies that $\psi$ is a cochain complex map 
and that the inverse of $\phi$ is $\psi$. 
\end{remark}

Now we can translate the Gerstenhaber bracket from the Hochschild cochain complex 
$C^{*+1}(A)$ to the complex $\bold{C'}$ using $\psi$ and $\phi$. 
To do so, it is enough to translate the operation $\circulito$. 

\begin{notation}
Given $n \geq 1$,  
let $\alpha \in Map(G^{\times n},K)$, $\beta \in Map(G^{\times m},K)$ 
and $y \in G$. 
For $i = 1,\dots, n$, denote $\gamma_i:=\gamma(\alpha,\beta,y)$ the map
$$\gamma_i:G^{n+m-1} \rightarrow K$$ such that 
$\gamma_i (x_1\, | \, \cdots \, | \, x_{n+m-1})=$
$$ 
\alpha(x_1\, | \, \cdots \, | \, x_i x_{i+1}\cdots x_{i+m-1}y \, | \, \cdots \, | \, x_{n+m-1})
\beta(x_i \, | \, \cdots \, | \, x_{i+m-1}).
$$
Denote
$$
\gamma^{(\beta,y)}_{\alpha}:= \sum_{i=1}^{n} (-1)^{(m-1)(i-1)} \, \gamma_i.
$$
\end{notation}

\begin{lemma} 
Let $\alpha \in Map(G^{\times n},K)$, $\beta \in Map(G^{\times m},K)$ and 
$x, y \in G$. Then 
$$
\phi_{n+m-1}(\, \psi_n(\alpha \otimes x) \, \circ \, \psi_{m}(\beta \otimes y) \,) =
\gamma^{(\beta,y)}_{\alpha}  \otimes xy. 
$$
\end{lemma}
\begin{proof} Fix $n \geq 1$ and let $i=1,\dots, n$. 
It is enough to show that 
$$
\phi_{n+m-1}(\, \psi_n(\alpha \otimes x) \, \circ_i \, \psi_{m}(\beta \otimes y) \,) = 
\gamma_i  \otimes xy 
$$
Let us remark that the left hand side  is equal to 
$$
\displaystyle{\sum_{c \in G} \varepsilon^{n+m-1}(f^n_{(\alpha \otimes x)} \circ_i f^m_{(\beta \otimes y)}, c) \otimes c}. 
$$
Let $\underline{x}=[\, x_1 \, | \, \cdots \, | \, x_{n+m-1}\,] \in G^{\times n+m-1}$, then for each $c \in G$
\[
\begin{array}{clc}
\varepsilon^{n+m-1}(f^n_{(\alpha \otimes x)} \circ_i f^m_{(\beta \otimes y)}, c) (\overline{x})&=& 
\lambda_{x_1\cdots x_{n+m-1}c} \,( f^n_{(\alpha \otimes x)} \circ_i f^m_{(\beta \otimes y)} (\overline{x}) \, ).\\
\end{array}
\]
Since $f^n_{(\alpha \otimes x)} \circ_i f^m_{(\beta \otimes y)} (\overline{x})$ 
\[
\scalebox{.9}{
$
\begin{array}{cl}
=&
f^n_{(\alpha \otimes x)} ( x_1 \otimes \cdots \otimes 
f^m_{(\beta \otimes y)} (x_i \otimes \cdots \otimes x_{i+m-1}) \otimes \cdots \otimes x_{n+m-1} )\\
=&\beta[\, x_i \, | \, \cdots \, | \, x_{i+m-1}\,]
f^n_{(\alpha \otimes x)} ( x_1 \otimes \cdots \otimes x_i\cdots x_{i+m-1} y \otimes \cdots \otimes x_{n+m-1} ) \\
= &\beta[\, x_i \, | \, \cdots \, | \, x_{i+m-1}\,] 
\alpha[\, x_1 \,|\, \cdots \, | \, x_i\cdots x_{i+m-1} y \,| \, \cdots \, | x_{n+m-1} \,] \,
x_1\cdots x_{i+m-1} y x_{i+m} \cdots x_{n+m-1} x  \, 
\end{array}
$}
\]
then $\varepsilon^{n+m-1}(f^n_{(\alpha \otimes x)} \circ_i f^m_{(\beta \otimes y)}, c) (\overline{x})$
is equal to 
$$\beta[\, x_i \, | \, \cdots \, | \, x_{i+m-1}\,] 
\alpha[\, x_1 \,|\, \cdots \, | \, x_i\cdots x_{i+m-1} y \,| \, \cdots \, | x_{n+m-1} \,]$$
if 
when $c=xy$ (since $G$ is commutative) and zero otherwise. This proves what we wanted. 
\end{proof}

\begin{proposition}\label{bracket}
Consider the bilinear map
\[
[\, - \, , \, - \, ]: Map(G^{\times n},K) \otimes A \times Map(G^{\times m},K) \otimes A \longrightarrow 
Map(G^{\times n+m-1},K) \otimes A
\]
given by
$$
[\, \alpha \otimes x \,, \, \beta \otimes y ]= (\gamma^{(\beta,y)}_{\alpha} - (-1)^{(n-1)(m-1)} \gamma^{(\alpha,x)}_\beta ) \otimes xy 
$$
where $x,y \in G$, 
$\alpha \in Map(G^{\times n},K)$, $\beta \in Map(G^{\times m},K)$
and $\gamma^{(\beta,y)}_{\alpha}$ and $\gamma^{(\alpha,x)}_\beta$ are defined as above. 
Then the map $[\, - \, , \, - \, ]$ is well defined on ${H^{*+1}(G,K) \otimes A}$ 
and $(\,H^{*+1}(G,K) \otimes A, [\, - \, , \, - \, ] \,)$ is a graded Lie algebra. 

Moreover, $HH^{*+1}(A)$ with the Gerstenhaber bracket is isomorphic as a graded Lie algebra to 
 $(\,H^{*+1}(G,K) \otimes A, [\, - \, , \, - \, ] \,)$.  
\end{proposition}

\begin{proof}
It is a direct consequence of the following equality  
$$[\, \alpha \otimes x \, , \, \beta \otimes y \, ]= 
\phi_{n+m-1}[\, \psi_n(\alpha \otimes x) \, , \, \psi_{m}(\beta \otimes y) \,]_G.$$
\end{proof}

\section{ A basis for $H^*(G,K)$ with $G$ a cyclic group.}
The cohomology of a group $G$ with coefficients in the trivial $A$-module $K$ 
is the graded $K$-vector space 
$$H^*(G,K)=Ext^*_{A}(K,K).$$ 
The {\it{standard bar resolution}} can be used to compute the cohomology of G and 
is given by the following exact sequence: 
\[
\bold{B}:= \qquad 
\cdots 
\rightarrow
A[G^{\times n}]
\stackrel{d_{n}}{\longrightarrow} \, 
A[G^{\times n-1}] 
\rightarrow
\cdots  
\rightarrow \, 
A [G]  
\stackrel{d_1}{\longrightarrow}    \,
A[\,]
\stackrel{d_0}{\longrightarrow}   
K 
\rightarrow                    \, 
0
\]
where $A[G^{\times n}]$ is the free $A$-module with basis $G^{\times n}$, 
$d_0[\,]=1$, $d_1[\, x \, ]=x\, [\,]-[\,]$ and for $n > 1$
\[
\begin{array}{rcl}
d_{n} [\, x_1 \, | \,  \cdots \, | \,  x_{n}\, ] & = & 
x_1\, [\, x_2 \, | \,  \cdots \, | \,  x_{n}\, ] \\  
\, &+&
\displaystyle{\sum_{i=1}^{n-1} \, (-1)^{i} }
[\, x_1 \, | \,  \cdots \, | \,  x_ix_{i+1} \, | \,  \cdots \, | \,  x_{n}\, ] \\ 
\, & \, & + (-1)^{n}[\, x_1 \, | \,  \cdots  \, | \,  x_{n}\, ]
\end{array}
\] 
where $[\, x_1 \, | \,  \cdots \, | \,  x_{n+1}\, ] \in G^{\times n+1}$. 
After applying the functor $Hom_{A}(-,K)$ to the above resolution we obtain 
the complex $\bold{C}$ of section 1. 

From now on, let $G=<g>$ be the cyclic group of order $|G|$ 
with generator $g$. 
In order to compute the cohomology of the cyclic group, 
one can use the {\it{minimal projective resolution}}, 
which is given by the following exact sequence: 
\[
\bold{M}:= \: 
\cdots 
\rightarrow
A\,v_{n+1}
\stackrel{g-1}{\longrightarrow} \, 
A\,v_{n} 
\stackrel{T}{\longrightarrow} \, 
A\,v_{n-1} 
\stackrel{g-1}{\longrightarrow}   
\cdots
\stackrel{T}{\longrightarrow}   
A \,v_1
\stackrel{g-1}{\longrightarrow}  \, 
A \, v_0
\stackrel{\epsilon}{\longrightarrow}   
K 
\rightarrow                    \, 
0
\]
where $\epsilon$ is the augmentation map, 
$T=1+g+g^2+\cdots+g^{|G|-1} $, and  
the differentials are either multiplication by $g-1$ or $T$ depending on the degree.

Let $p$ be an odd prime. 
Assume from now on that the characteristic of $K$ is $p$ and divides $|G|$. 
The objective of this section is to present an explicit basis for $H^*(G,K)$. 
To do so we will give two comparison maps 
between the bar resolution and the minimal resolution, i.e. 
$q:\bold{B} \rightarrow \bold{M}$  and $s:\bold{M} \rightarrow \bold{B}.$
Therefore, we will define two families of morphisms of $A$-modules,
$q=(q_n)$ and $s=(s_n)$, such that the next diagrams commute:
\begin{equation}\label{diagrama}
\xymatrix{
A[G^{\times n+1}] \ar[d]_{q_{n+1}} \ar[r]^(.55){d_{n+1}} & 
A[G^{\times n}] \ar[d]_{q_n} \ar[r]^(.45){d_{n}} & 
A[G^{\times n-1}] \ar[d]_{q_{n-1}} \ar@{}[r] |{\cdots \cdots } & 
A[G] \ar[r]^{d_1} \ar[d]_{q_1} & 
 A[\,] \ar[d]_{q_0} \ar[r]^{d_{0}} 
& K \ar[d]_{id} \ar[r] & 0 \\
A\,v_{n+1}  \quad \ar[r]^(.55){g-1} \ar[d]_{s_{n+1}} &
A\,v_{n} \ar[r]^(.4)T \ar[d]_{s_n} &
A\,v_{n-1} \ar[d]_{s_{n-1}} \ar@{}[r] |{\cdots \cdots}&
A\,v_1 \ar[r]^{g-1} \ar[d]_{s_1}&
A\,v_0 \ar[r]^{\epsilon} \ar[d]_{s_0} 
& K \ar[d]_{id} \ar[r] & 0 \\
A[G^{\times n+1}] \ar[r]^(.55){d_{n+1}} & 
A[G^{\times n}]  \ar[r]^(.45){d_{n}}& 
A[G^{\times n-1}] \ar@{}[r] |{\cdots \cdots} & 
A[G] \ar[r]^{d_1} &
A[\,]  \ar[r]^{d_{0}} 
& K  \ar[r] & 0 
}
\end{equation}

\subsection*{Map from the bar resolution to the minimal.}
We will use the following notation to define the map 
$q:\bold{B} \rightarrow \bold{M}.$
\begin{notation}
Let $i$ be a positive integer. Denote $[\,i \,]_g$ to be the element in $A$ given by:
\[
[\,i \,]_g = \begin{cases} 1+g+ \cdots +g^{i-1} & \text{ if } i \geq 1 \\ 0 & \text{ if } i=0 \end{cases}
\]
If $x \in G$ and $x=g^i$ with $i=0,\dots,|G|-1$ 
then we denote $$[\, x \, ]_g:=[\, i \, ]_g$$. 
\end{notation}

\begin{remark}
Notice that $[g]_g=1$ and 
for any $x \in G$ $$ (g-1)[x]_g=x-1.$$
\end{remark}

\begin{notation}
Set a total order relation on the set $G$ as follows:
\[
1 < g < g^2 < \cdots < g^{|G|-1}.
\]
This is, if $x=g^i$ and $y=g^j$  where $i,j=0, \dots, |G|-1$, 
$x<y$ if and only if  $i < j$. 
Given $x \in G$, let
\[
Q(x):=\{ \,  y \in G \text{ such that } x\, y \,< \,y \, \}.
\] 
\end{notation}

\begin{remark}
Notice that $Q(1)=\o$, $Q(g)=\{ g^{p-1} \}$ and 
$Q(g^{p-1})=G \backslash \{1\}$. Moreover, 
if $x=g^i$ and $y=g^j$  where $i,j=0, \dots, |G|-1$, 
$x \in Q(y)$ if $i+j>|G|$
\end{remark}

\begin{definition}\label{cn}
Given $n\geq2$ and 
$\underline{x}=[\,x_1 \, | \, \cdots \, | \, x_i \, | \, x_{i+1} \, | \,\cdots\,| \, x_n \, ] \in G^{\times n}$. 
We will say that $\underline{x}$ {\it{ satisfies the condition}} $C(n)$ if and only if :
\begin{itemize}
\item $x_i \in Q(x_{i+1}) \; \forall \; i$  \underline{odd}, $i \geq 1$ when {\it{$n$ is even}} or  
\item $x_i \in Q(x_{i+1}) \; \forall \; i$  \underline{even} $i \geq 1$ when {\it{$n$ is odd}}. 
\end{itemize}
\end{definition} 

\begin{map}\label{qu}
Define $q:\bold{B} \rightarrow \bold{M}$ to be the family $(q_n)_{n \geq 0}$ of morphisms of $A$-modules given as follows. 
For $n=0$, $q_0: A \,[ \,] \rightarrow A  v_0$ is the map given by ${q_0 [\:]=v_0}$. 
Define $q_1: \, A \,[ \, G \,] \, \rightarrow A\, v_1$ to be the morphism such that $q_1[\, x \, ]=[\, x\, ]_g v_1 $. 
Let $n=2k$ be an integer with $k \geq 1$. 
Define  
\[ 
\begin{array}{rlllr}
q_n: & \, A \,[ \, G^{\times n}\,  ] \, &\longrightarrow &A\, v_n & \text{ and }\\ 
q_{n+1}:& \, A \,[ \, G^{\times n+1} \,] \, &\rightarrow &A\, v_{n+1} & \,
\end{array}
\]
to be the morphisms of $A$-modules such that 
\[
\begin{array}{rcl}
q_n \,(\underline{x}) & = & 
\begin{cases} 
v_n \qquad \qquad & \text { if } \underline{x} \text{ satisfies the condition } C(n) \\ 
0 & \text{ otherwise } 
\end{cases} \\
\quad & \quad & \quad \\
q_{n+1}\, (\underline{x}) &=&
\begin{cases} 
[\, x_1 \, ]_g \,v_{n+1}& \text { if } \underline{x} \text{ satisfies the condition } C(n+1) \\ 
0 & \text{ otherwise. } 
\end{cases}
\end{array}
\]
\end{map}

\subsection*{Map from the minimal resolution to the bar resolution.}
Next, we will define the map 
$
s:\bold{M} \longrightarrow \bold{B}
$.
To do so we will introduce some notation. 

\begin{notation}
Let $\underline{i}^k=(i_1, \dots, i_k)$ such that $i_r \in \{0, \dots, |G|-1\}$ for $r=1, \dots, k$. 
Define $a(\underline{i}^k)$  as follows:
\[
\begin{array}{rclc}
a(i)&:=&g^{|G|-1-i}&\text{ for } k=1 \\
a(i_1, \dots, i_{k+1})&:=&
a(i_1, \dots, i_k)g^{|G|-(i_{k+1}+k+1)}&\text{ for } k\geq1. 
\end{array}
\] 
\end{notation}

\begin{map}\label{ese}
Define $s:\bold{M} \longrightarrow \bold{B}$ to be the family $(s_n)_{n \geq 0}$ of morphisms of {$A$-modules} 
given as follows. 
For $n=0$, $s_0: \, A\, v_0\, \rightarrow  A \,[ \quad ]$ is $s_0(v_0)=[\quad]$. 
Define $s_1: \,  A\, v_1 \, \rightarrow  A \,[ \, G \,]$ to be the morphism such that $s_1(v_1)=[\, g \, ]$. 
Let $n=2k$ be an integer with $k \geq 1$. Define 
\[
\begin{array}{rlllr}
s_n: & \, A\, v_n\, &\longrightarrow  &A \,[  \, G^{\times \, n} \, ]& \text{ and }\\ 
s_{n+1}:& \,  A\, v_{n+1} \, &\rightarrow  &A \,[ \, G^{\times \, n+1} \,]& \,
\end{array}
\]
to be the morphisms of $A$-modules such that 
\[
\begin{array}{rcl}
s_n(v_n)&=&\displaystyle{\sum_{\underline{i}_k}}
a(\underline{i}^k) [\, g\, | g^{i_1} \, | \, g \, | \, g^{i_2} \, | \, \cdots  \, | \,   g  \, | \,g^{i_k} \, ]\\
\, &\, & \, \\
s_{n+1}(v_{n+1})&=&\displaystyle{\sum_{\underline{i}_k}}
a(\underline{i}^k) 
[\, g\, | g^{i_1} \, | \, g \, | \, g^{i_2} \, | \, \cdots  \, | \,   g  \, | \,g^{i_k}    \, | \, g  \,]
\end{array}
\]
where the sums are over 
all $k$-tuple $(i_1, \dots, i_k)$ of positive integers that vary from $0$ to $|G|-1$.
\end{map}


\subsection*{Explicit isomorphims between $H^{n}(G,K)$ and $K$.}
We will prove in the next two sections that the maps 
$
q:\bold{B} \rightarrow \bold{M}$ and 
$s:\bold{M} \rightarrow \bold{B}$
presented above are chain maps. 
Once we prove this, we will obtain for each $n$
explicit isomorphisms: 
$$
H^n(G,K) \stackrel{s_n}{\longrightarrow} K \stackrel{q_n}{\longrightarrow} H^n(G,K), 
$$
induced maps from the above chain maps.
Let us begin by $n=1$. 

\begin{map}\label{firstdegree}
If $\lambda \in K$, then $q_1(\lambda)$ 
is the derivation given by $q_1(\lambda)(x)=\epsilon[x]_g \, \lambda.$ 
Now, if $\alpha \in H^1(G,K)$, 
then $s_1(\alpha)=\alpha(g)$. 
\end{map}

\begin{map}\label{higherdegrees} 
If $\lambda \in K$ then $q_n(\lambda) \in H^n(G,K)$ is represented by the map 
\linebreak 
${q_n(\lambda): G^{\times n} \rightarrow K}$ defined as follows: 
$$
q_n(\lambda)(\underline{x})=\begin{cases} 
\lambda & 
\text{if $n$ is even and $\underline{x}$ satisfies the condition $C(n)$ } \\
\epsilon [\, x_1\,]_g \lambda & 
\text{if $n$ is odd and $\underline{x}$ satisfies the condition $C(n)$} \\
\\ 0 & \text{otherwise. }\end{cases}
$$ 

For $n\geq 1$, 
take $\alpha$ in $H^n(G,K)$ represented 
by a map ${\alpha:G^{\times n} \rightarrow K}$, 
then 
$$
s_n(\alpha)=
\begin{cases}
\displaystyle{\sum_{i_1=0}^{|G|-1}\cdots \sum_{i_k=0}^{|G|-1}} 
\alpha [\,g \, | \, g^{i_1} \, | \, \cdots \, | \, g \, | g^{i_k} \,] & \text{ if $n=2k$} \\
\, & \, \\
\displaystyle{\sum_{i_1=0}^{|G|-1}\cdots \sum_{i_k=0}^{|G|-1}} 
\alpha[\,g \, | \, g^{i_1} \, | \, g \, | \, \cdots \, | \, g^{i_k} \, | \, g\,] & \text{ if $n=2k+1$.}
\end{cases}
$$
\end{map}

Notice that $q_n(1)$ will provide a non trivial element of $H^n(G,K)$ 
and then we will be able to prove the following result:

\begin{teo}\label{basis}
Let $p$ be an odd prime. 
Denote $A=KG$ where the characteristic of $K$ is $p$ 
and $G$ is the cyclic group of order $|G|$ generated by $g$.  
Assume that $p$ divides $|G|$. 
Set $\beta^0=1_K$. 
Let $\beta^1$ be the derivation in $H^1(G,K)$ given by $\beta^1(x)=\epsilon[x]_g$. 
For $n > 1$, let $\beta^n$ be the element in $H^n(G,K)$ represented by the map 
$\beta^n: G^{\times n} \rightarrow K$ defined as follows 
$$
\beta^{n} (\underline{x})= \begin{cases} 1 & \text{if $n$ is even and \underline{x} satisfies $C(n)$} \\
\beta^1(x_1) &  \text{if $n$ is odd and \underline{x} satisfies $C(n)$}\\
0 & \text{otherwise} \end{cases}
$$
where $C(n)$ is the condition given in Definition \ref{cn}. 
Then the set $\{\beta^n\}_{n \geq 0}$ is a basis of the $k$-vector space $H^*(G,K)$. 
\end{teo}
In order to prove this result, we need to show first that $q$ is a comparison map, which is the purpose of the next section. 


\section{Proof of theorem \ref{basis}.}
In this section, we will show that  
$
q:\bold{B} \rightarrow \bold{P} 
$ 
is a chain map and then we will be able to prove Theorem \ref{basis}. 
To do so we will first give some properties of the set $Q(x)$. 
The proof of the following lemma is straightfoward.

\begin{lemma} \label{quantumintegers} 
For $x, y \in G$,  
\begin{itemize}
\item[$(i)$] $1 \notin Q(x)$, 
\item[$(ii)$] If $x\ne 1$ then $x^{-1} \in Q(x)$,  
\item[$(iii)$] $x \in Q(y)$ if and only if $y \in Q(x)$, 
\item[$(iv)$] If $x \in Q(y)$, then 
$$ T=x[\, y \,]_g - [\, xy \,]_g+[\, x\, ]_g, $$
\item[$(v)$] If $x \notin Q(y)$ then 
$$ 0=x[\, y \,]_g - [\, xy \,]_g+[\, x\,]_g. $$
\end{itemize}
\end{lemma}

\begin{proposition} \label{induccionuno}
Let $q_1$ and $q_2$ defined as above. Then 
$$q_{1} \, d_{2} =T\,q_{2}$$
\end{proposition}
\begin{proof}
Let $[\, x_1 \, | \, x_2\,]$ be an element of $G^{\times 2}$, 
then 
$$
\begin{array}{rcl}
q_1d_2[\, x_1 \, | \, x_2\,] &= &x_1\, q_1[\, x_2\,] - q_1[\, x_1x_2\,] +q_1[\, x_1\, ] \\
\, &= &(x_1[\, x_2\,]_g-[\, x_1x_2\,]_g+[\,x_1\,]_g)v_1.
\end{array}
$$ 
There are two cases. 
In the first one, we suppose $x_1 \in Q(x_2)$ 
then by Lemma \ref{quantumintegers} $(iv)$, 
$q_1d_2[\, x_1 \, | \, x_2\,]=T\, v_1$. 
In the second one, we suppose that $x_1 \notin Q(x_2)$,  
then by Lemma \ref{quantumintegers} $(v)$, $q_1d_2[\, x_1 \, | \, x_2\,]=0$. 
We conclude that $q_1d_2=T\,q_2$. 
\end{proof}

\begin{lemma}\label{auxiliar} 
Let $x$, $y$ and $z$ be elements of $G$. 
\begin{itemize}
\item[$(i)$] If $x \in Q(y)$ and $y \in Q(z)$ are such that 
$xy \ne 1$ and $yz \ne 1$ then 
\[
xy \in Q(z) \Leftrightarrow x \in  Q(yz) 
\]
\item[$(ii)$] If $x \in Q(y)$ and $y \in Q(z)$ are such that 
$xy =1$ and $yz \ne 1$ then 
\[
x \notin Q(yz).
\]
\item[$(iii)$] If $x \notin Q(y)$ and $y \in Q(z)$ then 
\[
xy \in Q(z) \text{ and } x \notin Q(yz).
\]
\item[$(iv)$] If $x \notin Q(y)$ and $y \notin Q(z)$ then 
\[
xy \in Q(z) \Leftrightarrow x \in Q(yz).
\]
\end{itemize}
\end{lemma}
\begin{proof}
Straightforward verification. 
\end{proof}

\begin{proposition} \label{inducciondos}
Let $q_2$ and $q_3$ be defined as above. Then 
$$q_{2} \, d_{3} =(g-1)\, q_{3}$$
\end{proposition}

\begin{proof}
Let $[\, x_1 \, | \, x_2\,| \, x_3\,]$ be an element of $G^{\times 3}$, 
then 
$$
\begin{array}{rcl}
q_2d_3[\, x_1 \, | \, x_2\, | \, x_3 \,] &=
&x_1\, q_2[\, x_2\, |\, x_3] - q_2[\, x_1x_2\, | \, x_3] + 
q_2[\, x_1 \, | \, x_2x_3 \,] -q_2[\, x_1\, | \, x_2 \, ]. \\
\end{array}
$$ 
Notice that if $x_i=1$ for any $i$ then $q_2d_3[\, x_1 \, | \, x_2\,| \, x_3\,]=0$ 
and the statement is true.  
So we will suppose that $x_i \ne 1$ for all $i$. 
The proof is divided into four cases.  
In the first one, we suppose that $x_1 \in Q(x_2)$ and $x_2 \in Q(x_3)$.  
Using $(i)$ and $(ii)$ from Lemma \ref{auxiliar}, 
$$q_2d_3[\, x_1 \, | \, x_2 \, | \, x_3\,]=
x_1\, q_2[\, x_2\, |\, x_3] -q_2[\, x_1\, | \, x_2 \, ]=x_1v_2-v_2$$ 
and the statement is true.  
Now, in the second case we suppose that $x_2 \in Q(x_3)$ 
but $x_1 \notin Q(x_2)$. Using $(iii)$ from Lemma \ref{auxiliar}, 
$$q_2d_3[\, x_1 \, | \, x_2 \, | \, x_3\,]=
x_1\, q_2[\, x_2\, |\, x_3] -q_2[\, x_1x_2\, | \, x_3 \, ]=x_1v_2-v_2$$ 
and the statement is also true for this case. 
The third case is when $x_1 \in Q(x_2)$ but $x_2$ is not in $Q(x_3)$. 
The statement follows from Lemma \ref{auxiliar} $(iii)$, 
$$q_2d_3[\, x_1 \, | \, x_2 \, | \, x_3\,]=
q_2[\, x_1\, |\, x_2x_3] -q_2[\, x_1| \, x_2 \, ]=0.$$ 
The last case to consider is when 
$x_1$ is not in $Q(x_2)$ and $x_2$ is not in $Q(x_3)$. 
Using $(iv)$ from Lemma \ref{auxiliar}, 
$$q_2d_3[\, x_1 \, | \, x_2 \, | \, x_3\,]=
q_2[\, x_1x_2\, |\, x_3] -q_2[\, x_1| \, x_2x_3 \, ]=0.$$ 
We conclude that $q_2d_3=(g-1)\,q_3$.
\end{proof}

\begin{notation} Let $n \in \bf{N}$, $n>2$. 
Let  $\underline{x}:=[\, x_1 \, | \, \cdots \, | \, x_n\,] \in G^{\times n}$. 
For $i=0,\dots,n$, we denote $r_i(\underline{x})$ the following element of 
$A[\, G^{\times n-1}\,]$: 
$$
\begin{array}{rcll}
r_0(\underline{x})&:=& x_1\, q_{n-1}[\, x_2 \, | \, \cdots \, | \, x_n \, ]&\, \\
r_i(\underline{x})&:=&
q_{n-1}[\,x_1 \, | \, \cdots \, | \,  x_{i}x_{i+1}   \, | \, \cdots \, | \,x_n\ \,] & 
\text{ for } i=1,\dots,n-1 \\
r_n(\underline{x})&:=&
q_{n-1}[\,x_1 \, | \, \cdots \, | \,x_{n-1}\ \,] & \, 
\end{array}
$$
\end{notation}

\begin{lemma}\label{ecuacionuno}
Let $n=2k$ with $k > 1$. Consider the maps $q_{n-1}$, $q_{n}$ and $q_{n+1}$ defined above.
\begin{itemize}
\item[$(i)$]  Let  ${\underline{x}:=[\, x_1 \, | \, \cdots \, | \, x_n\,] \in G^{\times n}}$. 
If ${x_i \in Q(x_{i+1})}$ for $i=2,\dots, n-1$ then 
$$q_{n-1}d_n(\underline{x}) =
r_{0}(\underline{x})-r_1(\underline{x})+r_n(\underline{x}) = 
\begin{cases} T \,v_{n-1} & \text{ if  $x_1 \in Q(x_2)$}  \\ 0 & \text{ if  $x_1 \notin Q(x_2)$. } \end{cases}$$
\item[$(ii)$] Let  ${\underline{x}:=[\, x_1 \, | \, \cdots \, | \, x_{n+1}\,] \in G^{\times n+1}}$. 
If $x_i \in Q(x_{i+1})$ for $i=1,\dots, n$ then 
$$q_{n}d_{n+1}(\underline{x}) =r_0(\underline{x})-r_n(\underline{x})=(x_1-1)v_{n}$$
\end{itemize}
\end{lemma}

\begin{proof}
$(i)$ We will compute $q_{n-1} d_{n}(\underline{x})$ using that  
$q_{n-1} d_{n}(\underline{x})=\sum_{i=0}^n (-1)^i \,r_i(\underline{x})$. 
Applying $(i)$ and $(ii)$ from lemma \ref{auxiliar}, notice that 
$$r_i(\underline{x}) \ne 0 \text{ if and only if  } r_{i+1}(\underline{x}) \ne 0$$ 
for all even integer $i=2,\dots,n-2$. Therefore,
\[
\begin{array}{rcl}
q_{n-1}d_n(\underline{x}) &=&r_0(\underline{x})-r_1(\underline{x})+r_n(\underline{x})\\
\, &=&x_1\,[\, x_2\,]_gv_{n-1} -[\, x_1x_2 \,]_gv_{n-1}+[\,x_1 \,]_gv_{n-1}. \\
\end{array}
\]
We apply lemma \ref{quantumintegers} to obtain that $q_{n-1}d_n(\underline{x})$ equals $T\,v_{n-1}$ if 
$x_1\in Q(x_2)$ and is zero otherwise. 
$(ii)$ The proof is similar to the proof of $(i)$
\end{proof}

\begin{lemma}\label{ecuaciondos}
Let $n=2k$ with $k > 1$. 
Consider the maps $q_{n-1}$, $q_{n}$ and $q_{n+1}$ defined above.  
\begin{itemize}
\item[$(i)$]Let  $\underline{x}:=[\, x_1 \, | \, \cdots \, | \, x_n\,] \in G^{\times n}$. 
Suppose that $x_i \in Q(x_i+1)$ for all odd integers $1 < i <n-1$ and that 
there exist some even integer $i$ such that 
$x_i \notin Q(x_{i+1})$. 
Denote by $j$ the smallest even integer such that $x_j$ is not in $Q(x_{j+1})$. 
Then 
$$ 
q_{n-1}d_n(\underline{x}) = r_0(\underline{x}) -
r_1(\underline{x})+r_j(\underline{x})= 
\begin{cases} T \,v_{n-1} & \text{ if  $x_1 \in Q(x_2)$}  \\ 0 & \text{ if  $x_1 \notin Q(x_2)$. } \end{cases}$$
\item[$(ii)$]Let  $\underline{x}:=[\, x_1 \, | \, \cdots \, | \, x_{n+1}\,]$ be in $G^{\times n+1}$. 
Suppose that $x_i \in Q(x_i+1)$ for all even integers $1 < i <n+1$ and that 
there exist some odd integer $i$ such that 
$x_i \notin Q(x_{i+1})$. 
Denote by $j$ the smallest odd integer such that $x_j$ is not in $Q(x_{j+1})$. 
Then 
$$ 
q_{n}d_{n+1}(\underline{x}) = r_0(\underline{x}) - r_j(\underline{x})=(x_1-1)v_n.
$$
\end{itemize}
\end{lemma}

\begin{proof}
We will compute $q_{n-1} d_{n}(\underline{x})$ using that 
$q_{n-1} d_{n}(\underline{x})=\sum_{i=0}^n (-1)^i \,r_i(\underline{x})$. 
Notice that $r_i(\underline{x})=0$ for all $i>j+1$ because of the definition of $q_{n-1}$. 
Moreover, $r_{j+1}(\underline{x})$ is also zero 
since $x_j$ is not in $Q(x_{j+1}x_{j+2})$ 
by $(iii)$ from lemma \ref{auxiliar}. 
Suppose that $j=2$. In this case $r_2(\underline{x}) \ne 0$ because 
$x_2x_3 \in Q(x_4)$ (use again $(iii)$ from lemma \ref{auxiliar}) and 
$x_i \in Q(x_{i+1})$ for odd integers $i>3$. 
Therefore, 
$$ 
q_{n-1}d_n(\underline{x}) = r_0(\underline{x}) - 
r_1(\underline{x})+r_2(\underline{x}).
$$
Suppose now that $j>2$. 
Notice that $r_j(\underline{x}) \ne 0$ because:  
(a) $x_i \in Q(x_{i+1})$ for all even integers $i$ such that $2 \leq i < j $;
(b) $x_jx_{j+1} \in Q(x_{j+2})$ because of $(iii)$ from lemma \ref{auxiliar}; 
and (c) $x_i \in Q(x_{i+1})$ for odd integers $i>j+1$.
Moreover, applying $(i)$ and $(ii)$ from lemma \ref{auxiliar}, 
notice that 
$r_i(\underline{x}) \ne 0$ if and only if 
$r_{i+1}(\underline{x}) \ne 0$, for all even integers $i < j$. 
Therefore, 
$$ 
\begin{array}{rcl}
q_{n-1}d_n(\underline{x}) &= &r_0(\underline{x}) -
r_1(\underline{x})+(-1)^jr_j(\underline{x}) \\
\, &=&x_1\,[\, x_2\,]_gv_{n-1} -[\, x_1x_2 \,]_gv_{n-1}+[\,x_1 \,]_gv_{n-1}.
\end{array}
$$
Then we apply lemma \ref{quantumintegers} 
and we obtain that $q_{n-1}d_n(\underline{x})$ is $T\,v_{n-1}$ if 
$x_1 \in Q(x_2)$ and zero otherwise. 
$(ii)$ The proof is similar to the proof of $(i)$. 
\end{proof}

\begin{proposition} 
The map 
$
q:\bold{B} \rightarrow \bold{M} 
$ 
defined in \ref{qu} is a chain map. 
\end{proposition}
\begin{proof}
First of all, notice that $\epsilon \, q_0 =d_0 $ and $q_0 \, d_1    =(g-1)\, q_1$. 
Now, let $n=2k$  with $k \geq 1$. Consider the maps $q_{n-1}$, $q_{n}$ and $q_{n+1}$, 
we will prove by induction on $k$ that 
\begin{itemize}
\item[$(i)$] $q_{n-1}\,d_{n} =T\,q_{n}$ 
\item[$(ii)$] $q_n d_{n+1} =(g-1)\, q_{n+1}$. 
\end{itemize}
For $k=1$, we consider the maps $q_1$, $q_2$ and $q_3$. 
Notice that for these maps, the assertion is true because 
of Propositions \ref{induccionuno} and \ref{inducciondos}. 
Let $k > 1$ and suppose that $(i)$ and $(ii)$ are satisfied for the maps 
$q_{2k-1}$, $q_{2k}$ and $q_{2k+1}$. 
We set $n=2k+2$ and we will prove the statement 
for $q_{n-1}$, $q_{n}$ and $q_{n+1}$. 
Let $\underline{x}:=[\, x_1 \, | \, \cdots \, | \, x_n\,]$ be in $G^{\times n}$. 
We begin by showing that $d_{n} \, q_{n-1} =T\,q_{n}$. We consider two cases. 
In the first one, suppose that $x_i \in Q(x_i)$ for all odd integers $1 \leq i <n-1$. 
By Lemmas \ref{ecuacionuno} and \ref{ecuaciondos}, 
$d_{n} \, q_{n-1}\underline{x}=T\,q_{n}\underline{x}$. 
In the second case, suppose that for some odd integer
$1 \leq i  < n-1$, $x_i \notin Q(x_i)$. 
Let $L$ be the maximum of all odd integers between $1$ and $n$ 
such that $x_i \notin Q(x_i)$. If $L=1$ 
then $d_{n} \, q_{n-1} =0$ by Lemmas \ref{ecuacionuno} and \ref{ecuaciondos} 
and the statement $(i)$ is true. 
Suppose now that $L>1$, then $r_0(\underline{x})=0=r_1(\underline{x})$. 
Since all $r_i(\underline{x})$ are either zero or $[x_1]_gv_{n-1}$ for $i>1$, $d_{n} \, q_{n-1} =\,c\,[x_1]_gv_{n-1}$ where $c$ is an integer. 
By the inductive hypothesis, $d_{n} \, q_{n-1}$ 
must be in the image of the map given by the multiplication by $T$. 
We conclude then that the integer $c$ must be zero and the statement $(i)$ is true. 
A similar argument shows that the statement $(ii)$ is true. 
\end{proof}

\begin{proof}[Proof of Theorem \ref{basis}]
Since $q$ is a chain map, then $q_n$ induces an isomorphism between $H^n(G,K)$ and $K$, which we gave explicitly in \ref{firstdegree} and \ref{higherdegrees}. Therefore $q_n(1)$ is non zero and we deduce that $\{\beta_n\}_{n\geq 0}$ is a basis of $H^{*}(G,K)$.  
\end{proof}


\section{The chain map $s$.}
We will prove that the application 
$
s:\bold{M} \rightarrow \bold{B} 
$ 
defined in the section 2 is a chain map. 
To do so we need some technical results about the 
$a(i_1, \dots, i_k)$. 

\begin{lemma}\label{primero}
Let $(i_1, \dots, i_k)$ 
be a $k$-tuple of positive integers that vary from 0 to $|G|-1$. Then
\begin{itemize}
\item[$(i)$] $a(0)g=a(|G|-1)$ 
\item[$(ii)$] $a(0, i_2\dots,i_k)g=a(|G|-1,i_2\dots,i_k)$ 
\item[$(iii)$]$a(i)g=\:a(i-1)$ for $i\geq1$ 
\item[$(iv)$]$a(i_1,i_2,\dots,i_k)g=a(i_1-1,i_2\dots,i_k)$ for $i_1\geq1$

\end{itemize} 
\end{lemma}
\begin{proof}
$(i)$ Straightforward verification. 
$(ii)$ Follows by induction on $k$. 
$(iii)$ Straightforward verification. 
$(iv)$ We prove it by induction on $k$. 
Notice that the statement is true for $k=1$ because of $(iii)$. 
Now, let $k\geq1$ and suppose that $a(i_1,i_2,\dots,i_k)g=a(i_1-1,i_2\dots,i_k)$. 
We will show that the statement is true for $k+1$. By definition and under the inductive hypothesis 
\[
\begin{array}{rcl}
a(i_1,i_2,\dots,i_{k+1})g&=&a(i_1,i_2,\dots,i_k)g^{|G|-(i_{k+1}+k+1)}g  \\
\, &=&a(i_1-1,i_2,\dots,i_k)g^{|G|-(i_{k+1}+k+1)} \\
\, &=&a(i_1-1,i_2,\dots,i_{k+1})
\end{array}
\]
\end{proof}

\begin{notation}
For $r=1,\dots,k$, denote $e_r$ the $k$-tuple such that the $r$-th 
coordinate is $1$ and the other coordinates are zero. 
\end{notation}

\begin{lemma}\label{tercero}
Let $k\geq2$ and let $\underline{i}^k:=(i_1, \dots, i_k)$  
be a $k$-tuple of positive integers that vary from 0 to $|G|-1$. 
Then for $1 \leq j<k$ 
\begin{itemize}
\item[$(i)$] 
$a(\underline{i}^k - (i_j-|G|+1)e_j -i_{j+1}e_{j+1})=
a(\underline{i}^k -i_{j}e_{j}-(i_{j+1}-|G|+1)e_{j+1})$  
\item[$(ii)$] 
$a(\underline{i}^k -e_j-i_{j+1}e_{j+1})=a(\underline{i}^k -(i_{j+1}-|G|+1)e_{j+1})$ 
for $i_{j} \geq1$ 
\item[$(iii)$] 
$a(\underline{i}^k -i_je_j-e_{j+1})=a(\underline{i}^k -(i_j-|G|+1)e_j)$ 
for $i_{j+1} \geq1$ 
\item[$(iv)$] $a(\underline{i}^k -e_j)=a(\underline{i}^k -e_{j+1})$ 
for $i_j \geq 1$ and $i_{j+1} \geq 1$. 
\item[$(v)$] $a(\underline{i}^k - (i_k-|G|+1)e_k)=a(\underline{i}^k - i_ke_k)g$
\item[$(vi)$] $a(\underline{i}^k - e_k)=a(\underline{i}^k )g$
\end{itemize} 
\end{lemma}

\begin{proof}
We illustrate the proof of $(iv)$. The proof is given by induction on $k$. 
For the first step, let $k=2$ and $e_1$. 
Then for $i_1, i_2 \geq 1$  that 
\[
a(i_1-1,i_2)= a(i_1,i_2)g =a(i_1)g^{|G|-(i_{2}+2)}g=a(i_1)g^{|G|-(i_{2}-1+2)} =a(i_1,i_2+1).
\]
Next, let $k \geq 2$. Suppose that the assertion is true for $k$ and every $1 \leq l < k$.  
We will show that the statement is true for $k+1$. We introduce first some notation. 
Given $\underline{i}^{k+1}$, an $(k+1)$-tuple  
denote by $\underline{i}^{k,T}$ the $k$-tuple obtained 
by forgetting the last coordinate of $\underline{i}^{k+1}$. 
For $l<k$, notice that 
$$
\begin{array}{rcl}
a(\underline{i}^{k+1} -e_r)&=&a(\underline{i}^{k,T} -e_{r})g^{|G|-(i_{k+1}+k+1)}\\ 
\, &=&a(\underline{i}^{k,T} -e_{r+1})g^{|G|-(i_{k+1}+k+1)} \\
\, &=&a(\underline{i}^{k+1} -e_{r+1})
\end{array}
$$
by definition of $a(\underline{i}^{k+1})$ and inductive hypothesis. Now, let us suppose that $r=k$. Given $\underline{i}^{k+1}$ 
denote by $\underline{i}^{k-1,T}$ the $k$-tuple obtained 
by forgetting the last two coordinate of $\underline{i}^{k+1}$. Then  
$$
\begin{array}{rcl}
a(\underline{i}^{k+1} -e_k)&=&a(\underline{i}^{k,T}-e_{k})g^{|G|-(i_{k+1}+k+1)}\\ 
\, &=&a(\underline{i}^{k-1,T})g^{|G|-(i_k-1+k)}g^{|G|-(i_{k+1}+k+1)} \\
\, &=&a(\underline{i}^{k-1,T})g^{|G|-(i_k+k)}gg^{|G|-(i_{k+1}+k+1)} \\
\, &=&a(\underline{i}^{k,T})g^{|G|-(i_{k+1}-1+k+1)} \\
\, &=&a(\underline{i}^{k+1} -e_{k+1})
\end{array}
$$
Then the statement is true for every $1 \leq l < k+1$. 
\end{proof}

\begin{notation}
Fix $n=2k$ an even positive integer. 
Let $\underline{i}^k:=(i_1, \dots, i_k)$ 
be a $k$-tuple of positive integers that vary from $0$ to $|G|-1$. \\
Denote 
\[
g_n^{\underline{i}^k}:=[\, g \, | \,  g^{i_1} \, | \, g \, | \,  g^{i_2} \, | \, \cdots \, | \, g \, | \, g^{i_k} \,] \, \in G^{\times n}.
\]
Denote 
\[
g_{n+1}^{\underline{i}^k}:=[\, g \, | \,  g^{i_1} \, | \, g \, | \,  g^{i_2} \, | \, \cdots \, | \, g \, | \, g^{i_k} \, | \, g \,] \, \in G^{\times n+1}.
\]
For $j=0,\ldots,n$ denote 
\[
\begin{array}{rcl}
\sigma_0^n({\underline{i}^k})&:=&
g\,[ \,  g^{i_1} \, | \, g \, | \,  g^{i_2} \, | \, \cdots \, | \, g \, | \, g^{i_k} \,] \, ,  \\
\sigma_1^n({\underline{i}^k})&:=&
[\,  g^{i_1+1} \, | \, g \, | \,  g^{i_2} \, | \, \cdots \, | \, g \, | \, g^{i_k} \,] \, ,\\
\sigma_{j}^n(\underline{i}^k)&:=&
\begin{cases}
[\, g \, | \,  g^{i_1} \, | \, \cdots \, | \, g \, | \,  g^{i_j+1} \, | \, g^{i_{j+1}}\, | \, \cdots \, | \, g \, | \, g^{i_k} \,] 
& \text{ if $1< j < n$ is even} \\
[\, g \, | \,  g^{i_1} \, | \, \cdots \, | \, g \, | \,  g^{i_j} \, | \, g^{i_{j+1}+1}\, | \, \cdots \, | \, g \, | \, g^{i_k} \,]
& \text{ if $1< j < n$ is odd} \\
\end{cases} \\
\sigma_{n}^n(\underline{i}^k)&:=&
[\, g \, | \,  g^{i_1} \, | \, g \, | \,  g^{i_2} \, | \, \cdots \, | \, g^{i_{k-1}} | \, g \, ] \,.
\end{array}
\]
For $j=0,\ldots,n+1$
denote 
\[
\begin{array}{rcl}
\sigma_0^{n+1}({\underline{i}^k})&:=& 
g\,[ \,  g^{i_1} \, | \, g \, | \,  g^{i_2} \, | \, \cdots \, | \, g \, | \, g^{i_k} \,  | \, g \,] \, ,\\
\sigma_1^{n+1}({\underline{i}^k})&:=& 
[\,  g^{i_1+1} \, | \, g \, | \,  g^{i_2} \, | \, \cdots \, | \, g \, | \, g^{i_k} \,  | \, g \,] \, ,\\
\sigma_{j}^{n+1}({\underline{i}^k})&:=&
\begin{cases}
[\, g \, | \,  g^{i_1} \, | \, \cdots \, | \, g \, | \,  g^{i_j+1} \, | \, g^{i_{j+1}}\, | \, \cdots \, | \, g \, | \, g^{i_k} \, | \, g \,] 
& \text{ if $1<j<$ is even} \\
[\, g \, | \,  g^{i_1} \, | \, \cdots \, | \, g \, | \,  g^{i_j} \, | \, g^{i_{j+1}+1}\, | \, \cdots \, | \, g \, | \, g^{i_k} \, | \, g \,]
& \text{ if $1<j<n$ is odd} \\
\end{cases}\\
\sigma_{n}^{n+1}({\underline{i}^k})&:=&
[\, g \, | \,  g^{i_1} \, | \, \cdots \, | \, g \, | g^{i_{k-1}} \, | \, g^{i_k+1}] \, , \\
\sigma_{n+1}^{n+1}({\underline{i}^k})&:=&
[\, g \, | \,  g^{i_1} \, | \, g \, | \,  g^{i_2} \, | \, \cdots \, | \, g \, | g^{i_k}] \, .
\end{array}
\]
\end{notation}

\begin{remark}
Let $n=2k$. Clearly,  	 
\[
\begin{array}{rcl}
d_{n}(g_n^{\underline{i}^k})&=&
\displaystyle{\sum_{j=0}^n (-1)^i \, \sigma_j^n({\underline{i}^k})} \\
d_{n+1}(g_{n+1}^{\underline{i}^k})&=&
\displaystyle{\sum_{j=0}^{n+1} (-1)^i \, \sigma_j^{n+1}({\underline{i}^k}) }
\end{array}
\]

\end{remark}

\begin{lemma} Let $n=2k$ where $k\geq1$. Then 
\[
\begin{array}{clcl}
(i)&\displaystyle{\sum _{\underline{i}^k} }\, 
a({\underline{i}^k}) \, 
\big( \, \sigma_0^n({\underline{i}^k}) \, - \, 
\sigma_1^n({\underline{i}^k}) \,  \big)
&=&0 \\
(ii)&\displaystyle{\sum _{\underline{i}^k} }\, 
a({\underline{i}^k}) \, 
\big( \, \sigma_0^{n+1}({\underline{i}^k}) \, - \, 
\sigma_1^{n+1}({\underline{i}^k}) \,  \big) 
&=&0
\end{array}
\]
where the sum is over 
all $k$-tuple $\underline{i}^k=(i_1, \dots, i_k)$ of positive integers that vary from 0 to $|G|-1$. 
For $k >1$, let $j$ be an even positive integer 
$1 < j < n$. Then
\[
\begin{array}{clcl}
(iii)&\displaystyle{\sum _{\underline{i}^k} }\, 
a({\underline{i}^k}) \, 
\big( \, \sigma_j^n({\underline{i}^k}) \, - \, 
\sigma_{j+1}^n({\underline{i}^k}) \,  \big)
&=&0 \\
(iv)&\displaystyle{\sum _{\underline{i}^k} }\, 
a({\underline{i}^k}) \, 
\big( \, \sigma_j^{n+1}({\underline{i}^k}) \, - \, 
\sigma_{j+1}^{n+1}({\underline{i}^k}) \,  \big) 
&=&0
\end{array}
\]
where the sum is over 
all $k$-tuple $\underline{i}^k=(i_1, \dots, i_k)$ of positive integers that vary from 0 to $|G|-1$.
\end{lemma}

\begin{proof}$(i)$
If $n=2$ then 
$\sum _{i=0}^{|G|-1} \, a(i) \, \big( \, \sigma_0^2(i) \, - \, \sigma_1^2(i) \,  \big)$ is equal to 
\[
\big( \, a(0)g \, - \,  a(|G|-1) \, \big) \, [\,1\,] 
+ \displaystyle{\sum_{i=1}^{|G|-1}} \big( \, a(i)g \, - \,  a(i-1) \, \big) \, [\,g^{i}\,] 
\]
which is 0 because of Lemma \ref{primero}. 
Now, let $n>2$. Then  the sum 
$$\displaystyle{\sum_{i_1=0}^{|G|-1}  \, \sum_{i_2=0}^{|G|-1} \, \cdots \, \sum_{i_k=0}^{|G|-1}} 
a({\underline{i}^k}) \, 
\big( \, \sigma_0^n({\underline{i}^k}) \, - \, \sigma_1^n({\underline{i}^k}) \big) $$
is equal to 
\[
\begin{array}{c}
\displaystyle{\sum_{i_2=0}^{|G|-1} \, \cdots \, \sum_{i_k=0}^{|G|-1}} 
\, \big( \, a(0,i_2,\ldots,i_k)g - a(|G|-1,i_2,\dots,i_k) \, \big) [\, 1 \, | \, g  \, | \, g^{i_2} \, | \, \cdots  \, | \, g  \, | \, g^{i_k}\, ] \, + \,  \\
\, \\
\displaystyle{\sum_{i_1=1}^{|G|-1}  \, \sum_{i_2=0}^{|G|-1} \, \cdots \, \sum_{i_k=0}^{|G|-1}} 
\, \big( \, a(i_1,i_2,\ldots,i_k)g - a(i_1-1,i_2,\dots,i_k) \, \big) [\, g^{i_1} \, | \, g  \, | \, \cdots  \, | \, g  \, | \, g^{i_k}\, ]
\end{array}
\]
which is 0 because of Lemma \ref{primero}. 
$(ii)$ Analogue to the proof of $(i)$.
$(iii)$ Let us remark that 
\[
\displaystyle{\sum_{i_j=0}^{|G|-1} \sum_{i_{j+1}=0}^{|G|-1}} \, a({\underline{i}^k}) \,
\big( \, \sigma_j^n(\underline{i}^k) \, - \, \sigma_{j+1}^n({\underline{i}^k}) \, \big)
\]
is equal to
\[
\scalebox{.8}{$
\begin{array}{l}
\big(a(i_1,\ldots,\overbrace{|G|-1}^{j-th},0,\ldots,i_k)
-a(i_1,\ldots,0,\overbrace{|G|-1}^{(j+1)-th},\ldots,i_k)\big)\, 
\sigma_j^n
(i_1,\ldots,\overbrace{0}^{j-th}, \overbrace{0}^{(j+1)-th},\ldots,i_k) \\
+\displaystyle{\sum_{i_j=1}^{|G|-1}}
\big(a(i_1,\ldots,\overbrace{i_j-1}^{j-th},0,\ldots,i_k)
-a(i_1,\ldots,i_j,\overbrace{|G|-1}^{(j+1)-th},\ldots,i_k)\big)\, 
\sigma_j^n
(i_1,\ldots,\overbrace{0}^{j-th}, i_{j+1},\ldots,i_k)   \\
+\displaystyle{\sum_{i_{j+1}=1}^{|G|-1}}
\big(a(i_1,\ldots,\overbrace{|G|-1}^{j-th},i_{j+1},\ldots,i_k)
-a(i_1,\ldots,\overbrace{0}^{j-th},i_{j+1}-1,\ldots,i_k)\big)\, 
\sigma_j^n
(i_1,\ldots, i_{j},\overbrace{0}^{(j+1)-th},\ldots,i_k)   \\
+\displaystyle{\sum_{i_j=1}^{|G|-1} \sum_{i_{j+1}=1}^{|G|-1}} \, 
\big(a(i_1,\ldots,\overbrace{i_j-1}^{j-th},i_{j+1},\ldots,i_k)
-a(i_1,\ldots,\overbrace{i_j}^{j-th},i_{j+1}-1,\ldots,i_k)\big)\, 
\sigma_j^n (\underline{i}^k)
\end{array}
$}
\]
Using $(i)-(iv)$ from lemma \ref{tercero} we obtain that the above sum is 0.
$(iv)$ Analogue to the proof of $(iii)$
\end{proof}

\begin{lemma}Let $n=2k$ where $k >1$. Then 
\[
\displaystyle{\sum_{\underline{i}^k} } \, a({\underline{i}^k})\, 
\sigma_n^n({\underline{i}^k}) =
T 
\displaystyle{\sum_{\underline{i}^{k-1} } } \, 
a(\underline{i}^{k-1}) 
\sigma_{n}^{n}(\underline{i}^k)
\]
where the sum in the leftside is over 
all $k$-tuple $\underline{i}^k=(i_1, \dots, i_k)$ of positive integers that vary from 0 to $|G|-1$ and the 
sum in the rightside is over 
all $(k-1)$-tuple $\underline{i}^{k-1}=(i_1, \dots, i_{k-1})$ of positive integers that vary from 0 to $|G|-1$ 
\end{lemma}
\begin{proof}
Let us remark that 
$
\displaystyle{\sum_{\underline{i}^k} } \, a({\underline{i}^k})\, \sigma_n({\underline{i}^k}) 
$
is equal to 
$$
\displaystyle{\sum_{\underline{i}^{k-1}}  \sum_{i_k=0}^{|G|-1}} \, a({\underline{i}^{k-1}})\, g^{|G|-(i_k+k)} \, 
[ \, g\, | g^{i_1} \, | \, g \, | \, g^{i_2} \, | \, \cdots  \, | \,   g  \, | \,g^{i_{k-1}}    \, | \, g  \,]
$$
Since $\displaystyle{\sum_{i_k=0}^{|G|-1} } \, g^{q-(i_k+k)} = T$, we obtain what we wanted. 
\end{proof}

\begin{lemma}Let $n=2k$ where $k >1$. Then 
\[
\displaystyle{\sum_{\underline{i}^k} } \, a({\underline{i}^k})\, \big(
\sigma_n^{n+1}({\underline{i}^k}) - \sigma_{n+1}^{n+1}(\underline{i}^k) \big) =
(g-1)
\displaystyle{\sum_{\underline{i}^{k} } } \, 
a(\underline{i}^{k}) \sigma_{n+1}^{n+1}(\underline{i}^k)
\]
where the sum is over 
all $k$-tuple $\underline{i}^k=(i_1, \dots, i_k)$ of positive integers that vary from 0 to $|G|-1$. 
\end{lemma}
\begin{proof}
Let us remark that 
$
\displaystyle{\sum_{\underline{i}^k} } \, a({\underline{i}^k})\, \big(
\sigma_n^{n+1}({\underline{i}^k}) - \sigma_{n+1}^{n+1}(\underline{i}^k) \big) 
$
is equal to 
$$
\begin{array}{c}
\displaystyle{\sum_{i_1=0}^{|G|-1} \cdots   \sum_{i_{k-1}=0}^{|G|-1}} \, 
\big( a(i_1,\ldots,i_{k-1},|G|-1) - a(i_1,\ldots,i_{k-1},0) \big) [\, g \, | \, g^{i_1}\, | \, \cdots \, | \, g \, | \, 1\,] \, + \\
\displaystyle{\sum_{i_1=0}^{|G|-1} \cdots    \sum_{i_{k-1}=0}^{|G|-1} \sum_{i_k=1}^{|G|-1}} \, 
\big( a(i_1,\ldots,i_{k-1},i_k-1) - a(i_1,\ldots,i_{k-1},i_k) \big) [\, g \, | \, g^{i_1}\, | \, \cdots \, | \, g \, | \, g^{i_k}\,] \,  \\
\end{array}
$$
Using Lemma \ref{tercero} $(v)$ $(vi)$, the above sum is equal to 
$$
\begin{array}{c}
\displaystyle{\sum_{i_1=0}^{|G|-1} \cdots    \sum_{i_{k-1}=0}^{|G|-1} \sum_{i_k=0}^{|G|-1}} \, 
(g-1) a(i_1,\ldots,i_{k-1},i_k)  [\, g \, | \, g^{i_1}\, | \, \cdots \, | \, g \, | \, g^{i_k}\,] \,  \\
\end{array}
$$

\end{proof}

\begin{remark} Notice that  
\begin{itemize}
\item[$(i)$] $d_0\,s_0=\epsilon$ 
\item[$(ii)$] $d_1\,s_1=s_0 \, (g-1)$
\end{itemize}
\end{remark}

\begin{proposition} 
The map $s:\bold{M} \rightarrow \bold{B}$ defined in \ref{ese} is a chain map.  

\end{proposition}
\begin{proof}
Let $n=2k$ be an even integer where $k \geq 1$. 
Consider the maps $s_{n-1}$, $s_{n}$ and $s_{n+1}$ defined as above.  
Then $(i)$ $s_{n-1}T=d_ns_n$ and $(ii)$ $s_n(g-1)=d_{n+1}s_{n+1}$ as 
a consequence of the above lemmas. We conclude that $s$ is a chain map. 
\end{proof}

\begin{corollary}
The induced isomorphisms 
$$s_n:H^n(G,K) \rightarrow K \quad \text{ and }\quad q_n: K \rightarrow H^n(G,K)$$ 
defined in \ref{firstdegree} and \ref{higherdegrees} are inverse. 
\end{corollary}
\begin{proof}
Notice that $s_nq_n(1)=s_n(\beta^n)=1$. 
\end{proof}


\section{The Lie structure on $H^{*+1}(G,K) \otimes A$.}
In this section we will provide a formula that computes 
the Gerstenhaber bracket in $H^{*+1}(G,K) \otimes A$. 
We begin by giving an example of the computation in degree one, 
this is the computation of the commutator bracket. 

\begin{example}
For $i=0,\dots, |G|-1$ the derivation $D_i:A \rightarrow A$ is defined by ${D_i(g)=g^{i+1}}$. 
Let $\beta$ be the derivation on 
$H^1(G,K)$ defined as follows ${\beta(x)=\epsilon[\, x \,]_g}$. 
Recall that $q_1$ and $s_1$ 
denote the explicit isomorphisms between $H^1(G,K)$ 
and the field given in section 2 (map \ref{firstdegree}), 
so $q_1(1_K)=\beta$ and $s_1(\beta)=1_K$. 
It is clear that $H^1(G,K)$ is the vector space generated by $\beta$. 
Moreover, we can check that $$\phi_1(D_i)= \beta \otimes g^i$$ 
where $\phi_1$ is defined in the first section (map \ref{phi}). 
In fact, $\phi_1(D)=\beta \otimes D(g)$ for any derivation $D$ of $A$. 
Let $x,y$ be in $G$. We will compute $[\, \beta \otimes x \, , \, \beta \otimes y \,]$ as follows. 
From the formula of the commutator (proposition \ref{bracket}), we know that 
$$[\, \beta \otimes x \, , \, \beta \otimes y \,]=
(\gamma^{(\beta,y)}_{\beta} -  \gamma^{(\beta,x)}_\beta ) \otimes xy$$  
where $\gamma^{(\beta,y)}_{\beta}$ and $\gamma^{(\beta,x)}_\beta $ 
are elements in $H^1(G,K)$. Therefore both of them are multiples  
of the derivation $\beta$ since $H^1(G,K)$ is generated by $\beta$.   
Then ${(\gamma^{(\beta,y)}_{\beta} -  \gamma^{(\beta,x)}_\beta )}$ is 
$\lambda \beta$ where $\lambda$ is in $K$.  
In fact, $\lambda=s_1(\gamma^{\beta,y}_{\beta}-\gamma^{\beta,x}_{\beta})$. Then 
$$
\begin{array}{lcl}
s_1(\gamma^{\beta,y}_{\beta}-\gamma^{\beta,x}_{\beta})
&=&\gamma^{\beta,y}_{\beta}(g)-\gamma^{\beta,x}_{\beta}(g)\\
\,&=& \beta(gy)\beta(g)-\beta(gx)\beta(g)\\ 
\,&=& \beta(y)-\beta(x)\\ 
\end{array}
$$
since $\beta$ is a derivation and $\beta(g)=\epsilon[g]_g=1$. 
Therefore $$[\, \beta \otimes x \, , \, \beta \otimes y \,] = (\beta(y)-\beta(x)) \:\beta \otimes xy.$$
\end{example}

In the next paragraphs we will compute the Gerstenhaber bracket as we did with the commutator bracket.

\begin{notation} Let $n\geq 1$. We denote $\beta^n$ be the element in $H^n(G,K)$ represented by the map 
$q_n(1):G^{\times n} \rightarrow K$, i.e. given $\underline{x}:=[\, x_1 \, | \, \cdots \, | \, x_n\,]$ in $G^{\times n}$
$$
\beta^{n} (\underline{x})= \begin{cases} 1 & \text{if $n$ is even and \underline{x} satisfies $C(n)$} \\
\beta(x_1) &  \text{if $n$ is odd and \underline{x} satisfies $C(n)$}\\
0 & \text{otherwise} \end{cases}
$$
where $\beta$ is the derivation in $H^1(G,K)$ given by $\beta(x_1)=\epsilon[\, x_1\,]_g$.
\end{notation}

Let $x,y$ be in $G$. 
The purpose of this section is to compute 
$[\, \beta^{n} \otimes x \, , \, \beta^{m} \otimes y\, ]$. 
To do so we need to compute 
$$\gamma=\gamma^{(\beta^{m},y)}_{\beta^{n}}-\gamma^{(\beta^{n}, x)}_{\beta^{m}}$$
Since $\gamma=s(\gamma) \beta^{n+m-1}$, we must compute $s(\gamma)$. 
Observe that it is enough to calculate $s(\gamma^{(\beta^m,y)}_{\beta^n})$ in four cases:
\begin{itemize}
\item[-] $n,m$ are both odd 
\item[-] $n$ is odd and $m$ is even
\item[-] $n$ is even and $m$ is odd 
\item[-] $n,m$ are both even 
\end{itemize}
Recall that $\gamma^{(\beta^m,y)}_{\beta^n}=\sum_{i=1}^{n} (-1)^{(m-1)(i-1)} \, \gamma_i$ where 
$\gamma_i (x_1\, | \, \cdots \, | \, x_{n+m-1})$ equals to
$ 
\alpha(x_1\, | \, \cdots \, | \, x_i x_{i+1}\cdots x_{i+m-1}y \, | \, \cdots \, | \, x_{n+m-1})
\beta(x_i \, | \, \cdots \, | \, x_{i+m-1}).
$ 
In order to compute $s(\gamma^{(\beta^m,y)}_{\beta^n})$, we will compute $s(\gamma_i)$. 
  
\begin{lemma} Let $n$ and $m$ be  odd . Then 
$$s(\gamma^{(\beta^m,y)}_{\beta^n})=\beta(gy)$$
\end{lemma}
\begin{proof}
Suppose that $n=2k+1$ and $m=2h+1$, so $n+m-1=2(k+h)+1$. Then, 
$$
\begin{array}{lcl}
s(\gamma_1)&=& \displaystyle{\sum_{i_1=0}^{|G|-1} \cdots \sum_{i_{k+h}=0}^{|G|-1}}
\gamma_1 [\, g \, | g^{i_1} \, | \, g \, | \, \cdots \,|\, g^{i_{k+h}}\, | \, g  \, ] \\
\, &=& \displaystyle{\sum_{i_1=0}^{|G|-1} \cdots \sum_{i_{k+h}=0}^{|G|-1}}
\beta^n[\, g^{i_1+\cdots +i_{h}+h+1}y \, | \, g^{i_{h+1}} \, | \, g \, |  \, \cdots \, | \, g^{i_{k+h}}\, | \, g  \, ] \\
\, &\,&\qquad \qquad \qquad
\beta^m[\, g \, | g^{i_1} \, | \, g \, | \, \cdots \,|\, g^{i_{h}}\, | \, g  \,] \\
\, &\, & \, \\
\, &=& 
\beta^n[\, g^{h(|G|-1)+h+1} y\, | \, g^{|G|-1} \, | \, g \, |  \, \cdots \, | \, g^{|G|-1}\, | \, g  \, ]
\beta^m[\, g \, | g^{|G|-1} \, | \, g \, | \, \cdots \,|\, g^{|G|-1}\, | \, g  \,] \\
\, &=& 
\beta^n[\, gy \, | \, g^{|G|-1} \, | \, g \, |  \, \cdots \, | \, g^{|G|-1}\, | \, g  \, ]
\beta(g) \\
\, &=&  
\beta^n[\, gy \, | \, g^{|G|-1} \, | \, g \, |  \, \cdots \, | \, g^{|G|-1}\, | \, g  \, ] \\
\, &=&
\epsilon[gy]_g=\beta(gy).
\end{array}
$$
We carry out similar computations to get $s(\gamma_j)$ for $j=2,\dots,n$. We obtain that 
$s(\gamma_j)=\beta(gy)$ for $j$ odd and $s(\gamma_j)=-\beta(gy)$ for $j$ even. 
Therefore it is easy to deduce that $s(\gamma^{(\beta^m,y)}_{\beta^n})=\beta(gy)$. 
\end{proof}
\begin{lemma} Let $n$ be odd and $m$ be even. Then 
$$s(\gamma^{(\beta^m,y)}_{\beta^n})=\beta(y)$$
\end{lemma}
\begin{proof}
Suppose that $n=2k+1$ and $m=2h$, so $n+m-1=2(k+h)$. 
Let us compute $s(\gamma_j)$ where $j$ is even. 
$$
\begin{array}{lcl}
s(\gamma_j)&=& \displaystyle{\sum_{i_1=0}^{|G|-1} \cdots \sum_{i_{k+h}=0}^{|G|-1}}
\gamma_j [\, g \, |\, g^{i_1} \, | \, \cdots \, | \, g \, | \, g^{i_j} \,| \, \cdots \, | \, g \, |\, g^{i_{k+h}}\,  ] \\
\, &=& \displaystyle{\sum_{i_1=0}^{|G|-1} \cdots \sum_{i_{k+h}=0}^{|G|-1}}
\beta^n [\, g \, | \, g^{i_1} \, | \, g \, | \, \cdots \, | \,  g^{i_j+\cdots+i_{j+h-1}+h}y \, | \, g^{i_{j+h}} \, | 
\, \cdots \, |  \, g \, |\, g^{i_{k+h}}\, ]
\\
\, & \, &
\qquad \qquad \qquad  \beta^m [\, g^{i_j} \, | \, g \, | \, \cdots \, | \, g^{i_{j+h-1}}\, | \, g\,] \\
\, &=& \displaystyle{\sum_{i_{j+h}=0}^{|G|-1}}
\beta^n [\, g \, | \, g^{|G|-1} \, | \, g \, | \, \cdots \, | \,  g^{h(|G|-1)+h}y \, | \, g^{i_{j+h}} \, | \, \cdots \, |  \, g \, |\, g^{|G|-1}\, ]
\\
\, & \, &
\qquad \qquad \qquad 
\beta^m [\, g^{|G|-1} \, | \, g \, | \, \cdots \, | \, g^{|G|-1}\, | \, g\,] \\
\, &=&\displaystyle{\sum_{i_{j+h}=0}^{|G|-1}}
\beta^n [\, g \, | \, g^{|G|-1} \, | \, g \, | \, \cdots \, | \,  y \, | \, g^{i_{j+h}} \, | \, \cdots \, |  \, g \, |\, g^{|G|-1}\, ] \\ 
\end{array}
$$
which is zero if $y=1$. Suppose that $y \ne 1$ then 
$$
\begin{array}{lcl}
s(\gamma_j) &=&\displaystyle{\sum_{i_{j+h}=|G|-\beta(y)}^{|G|-1}}
\beta^n [\, g \, | \, g^{|G|-1} \, | \, g \, | \, \cdots \, | \,  y \, | \, g^{i_{j+h}} \, | \, \cdots \, |  \, g \, |\, g^{|G|-1}\, ] \\ 
\, &=&\beta(y)
\end{array}
$$
We perform a similar computation for $j$ odd and we obtain that $s(\gamma_j)=\beta(y)$. 
In this case $s(\gamma^{(\beta^m,y)}_{\beta^n})=\beta(y)$.
\end{proof}
\begin{lemma} Let $n$ be even and $m$ be odd. Then 
$$s(\gamma^{(\beta^m,y)}_{\beta^n})=0$$
\end{lemma}
\begin{proof}
Suppose that $n=2k$ and $m=2h+1$, so $n+m-1=2(k+h)$. 
For $j$ even, let us calculate $s(\gamma_j)$ . 
$$
\begin{array}{lcl}
s(\gamma_j)&=& \displaystyle{\sum_{i_1=0}^{|G|-1} \cdots \sum_{i_{k+h}=0}^{|G|-1}}
\gamma_1 [\, g \, |\, g^{i_1} \, | \, \cdots \, | \, g \, | \, g^{i_j} \,| \, \cdots \, | \, g \, |\, g^{i_{k+h}}\,  ] \\
\, &=& \displaystyle{\sum_{i_1=0}^{|G|-1} \cdots \sum_{i_{k+h}=0}^{|G|-1}}
\beta^n [\, g \, |\, g^{i_1} \, | \, \cdots \, | \, g \, | \, g^{i_j + \cdots i_{j+h} + h }y \, | \, g \, | \, g^{i_{j+h+1}}\, 
| \, \cdots \, | \, g \, |\, g^{i_{k+h}}\,  ]
\\
\, & \, &
\qquad \qquad \qquad  \beta^m 
[\, g^{i_j} \,| \, g \, | \, g^{i_{j+1}} \, | \, \cdots \, | \, g \, |\, g^{i_{j+h}}\,  ]  \\
\, &=& \displaystyle{\sum_{i_j=0}^{|G|-1} }
\beta^n [\, g \, |\, g^{|G|-1} \, | \, \cdots \, | \, g \, | \, g^{i_j + (|G|-1)h+h}  y \, | \, g \, | \, g^{|G|-1}\, 
| \, \cdots \, | \, g \, |\, g^{|G|-1}\,  ]
\\
\, & \, &
\qquad \quad  \beta^m 
[\, g^{i_j} \,| \, g \, | \, g^{|G|-1} \, | \, \cdots \, | \, g \, |\, g^{|G|-1}\,  ] \\
\, &=&\displaystyle{\sum_{i_j=0}^{|G|-1} }
\beta^n [\, g \, |\, g^{|G|-1} \, | \, \cdots \, | \, g \, | \, g^{i_j}  y \, | \, g \, | \, g^{|G|-1}\, 
| \, \cdots \, | \, g \, |\, g^{|G|-1}\,  ] \beta(g^{i_j}) \\
\, &=&\epsilon[g^{|G|-1}y^{-1}]_g=\beta(g^{|G|-1}y^{-1})
\end{array}
$$
We have that 
$0=\beta(1)=\beta(gy(g^{|G|-1}y^{-1}))=\beta(gy)+\beta(g^{|G|-1}y^{-1})$ since $\beta$ is a derivation. 
Therefore $s(\gamma_j)=-\beta(gy)$. 
A straightfoward computation for $s(\gamma_j)$ with $j$ odd gives that $s(\gamma_j)=\beta(gy)$. 
Then $s(\gamma^{(\beta^m,y)}_{\beta^n})=0$. 
\end{proof}
\begin{lemma} Let $n$ and $m$ be even . Then 
$$s(\gamma^{(\beta^m,y)}_{\beta^n})=0$$
\end{lemma}
\begin{proof}
Suppose that $n=2k$ and $m=2h$, so $n+m-1=2(k+h-1)+1$. 
We determine $s(\gamma_j)$ 
for $j\ne 1$ odd. 
$$
\begin{array}{lcl}
s(\gamma_j)&=& \displaystyle{\sum_{i_1=0}^{|G|-1} \cdots \sum_{i_{k+h-1}=0}^{|G|-1}}
\gamma_j [\, g \, | g^{i_1} \, | \, \cdots \,| \, g \, | \, g^{i_j} \, | \, \cdots \,| \, g \,  | \, g^{i_{k+h-1}}\,  ] \\
\,&=& \displaystyle{\sum_{i_1=0}^{|G|-1} \cdots \sum_{i_{k+h-1}=0}^{|G|-1}} 
\beta^n
[\,g\, | \, g^{i_1} \, | \, \cdots \, | \, g^{i_{j}+\cdots+i_{j+h-1}+h} y \, | \, g \, | \, \cdots \, | \, g^{i_{k+h-1}} \, | \, g \,  \,] \\
\, & \, &
\qquad \qquad \qquad 
\beta^m[\, g \, | \, g^{i_{j}} \, | \, \cdots \,  | \, g \, | \, g^{i_{j+h-1}} \, ] \\ 
\,&=& 
\beta^n
[\,g\, | \, g^{|G|-1} \, | \, \cdots \, | \, g^{h(|G|-1)+h} y \, | \, g \, | \, \cdots \, | \, g^{|G|-1} \, | \, g \,  \,] 
\beta^m[\, g \, | \, g^{|G|-1} \, | \, \cdots \,  | \, g \, | \, g^{|G|-1} \, ] \\ 
\, &=&
\beta^n
[\,g\, | \, g^{|G|-1} \, | \, \cdots \, | \, y \, | \, g \, | \, \cdots \, | \, g^{|G|-1} \, | \, g \,  \,] 
\end{array}
$$
which is $1$ if $y=g^{|G|-1}=g^{-1}$ and zero otherwise. 
Then we compute $s(\gamma_1)$  and $s(\gamma_j)$ for $j$ even and we obtain that in any case 
$s(\gamma_j)=1$ if $y=g^{-1}$ and zero otherwise.  
Clearly $s(\gamma^{(\beta^m,y)}_{\beta^n})=0$ if $y$ is not $g^{-1}$. 
In case $y=g^{-1}$ then 
$s(\gamma^{(\beta^m,y)}_{\beta^n})=\sum_{i=1}^n (-1)^{(i-1)} 1=0$. 
\end{proof}

\begin{teo}\label{gerstenhaber}
Let $p$ be an odd prime. 
Denote $A=KG$ where the characteristic of $K$ is $p$ 
and $G$ is the cyclic group of order $|G|$.  
Assume that $p$ divides $|G|$. 
Let
$$\varphi: H^*(G,K) \times G \rightarrow K$$ 
be the map which is linear in the first variable and additive in the second variable such that 
$$\varphi(\beta^n,x)=\begin{cases} \beta(x) & \text{ if $n$ is odd} \\ 0 & \text{ otherwise} \end{cases}$$ 
where $\beta^n$ is an element of the basis of $H^*(G,K)$ described above and $x$ is in $G$. 
Then the Gerstenhaber bracket on $H^*(G,K) \otimes A$ is given by 
\[
[\, \beta^n \otimes x \,, \, \beta^m \otimes y ]=  (\varphi(\beta^n,y) - \varphi(\beta^m,x)) \beta^{n+m-1} \otimes xy\] 
where $x$ and $y$ are in $G$.
\end{teo}

\begin{proof}
We will use Proposition \ref{bracket} in order to compute the Gerstenhaber bracket on $H^*(G,K) \otimes A$ . 
Notice that $\gamma^{(\beta^m,y)}_{\beta^n}=s(\gamma^{(\beta^m,y)}_{\beta^n}) \beta^{n+m-1}$ because of the Theorem $\ref{basis}$. 
The following table summarize the computations for $s(\gamma^{(\beta^m,y)}_{\beta^n})$ given by the above lemmas: 
\begin{center}
\begin{tabular}{| c | c | c |}
\hline
$n\setminus m$ & odd & even \\
\hline 
odd & $\beta(gy)$ & $\beta(y)$ \\
\hline 
even & $0$ & $0$ \\
\hline
\end{tabular}
\end{center}
We deduce from this table the above formula for the Gerstenhaber bracket. 
\end{proof}

Let $B$ be an associative algebra with unit and 
$\W:=Der(B)$ the Lie algebra of its derivations. 
We consider two modules over $\W$, the {\it{adjoint}} and the {\it{standard}}. 
The adjoint module is given by the adjoint map, i.e. D.D'=DD'-D'D and the 
standard module is given by the evaluation map, i.e. $D.x=D(x)$.

\begin{corollary}
Consider $HH^n(A)$ as a Lie module over $HH^1(A)$. 
\begin{itemize}
\item[$(i)$] $HH^n(A)$ is the adjoint module, if $n$ is odd. 
\item[$(ii)$] $HH^n(A)$ is the standard Lie module, if $n$ is even. 
\end{itemize}
\end{corollary}

\begin{proof}
Let $x=g^i$ with $i=0,\dots,|G|-1$. Notice that $\varphi(\beta^n,g^i)=i$ if $n$ is odd and $0$ otherwise. Then, 
$$
[\,\beta \otimes g^i \,,\, \beta^n \otimes g^j \,]=\begin{cases} (j-i) \beta^n \otimes g^{i+j} & \text{ if $n$ is odd} \\ 
\: j \qquad \beta^{n} \otimes g^{i+j} & \text{ if $n$ is even.} \end{cases}
$$
We identify $\beta \otimes g^i$ to the derivation $D_i$. For $(i)$ we associate $\beta^n \otimes g^j$ to the derivation $D_j$ 
and for $(ii)$ we associate $\beta^n \otimes g^j$ to the element $g^j$. The above computation proves the statement. 
\end{proof}

\begin{corollary}
Consider the Lie algebra $HH^{odd}(A)$. Then the Lie algebra $HH^1(A) \otimes K[t]$ given by the backet 
$
[\, D_i \otimes t^r \, , \, D_j \otimes t^s \,]= [\, D_i \, , \, D_j \,] \otimes t^{r+s}
$
is isomorphic to $HH^{odd}(A)$. 
\end{corollary}

\begin{proof}
The linear map from $HH^{odd}(A)$ to $HH^1(A) \otimes K[t]$ given by 
$$\beta^{2r+1} \otimes g^i \mapsto D_i \otimes t^r$$
is a map of Lie algebras because 
$$
[\, \beta^{2r+1} \otimes g^i \, , \, \beta^{2s+1} \otimes g^j \,]= (j-i) \beta^{2(r+s)+1} \otimes g^{i+j} 
$$
\end{proof}

\begin{remark}
The Lie algebra $HH^{odd}(A)$ is a Witt-type algebra (defined in \cite{yu})  
by considering: 
\begin{itemize}
\item[(i)] $\Gamma=H^{odd}(G,K) \otimes A$ to be the additive group given by the following multiplication: 
$$
(\beta^{n} \otimes x) \, \boxplus \, (\beta^m \otimes y) \, := \beta^{n+m-1} \otimes xy.
$$
\item[$(ii)$] The map $\varphi: \Gamma \rightarrow K$ given as $\varphi(\beta^n, x)=\beta(x)$. 
\end{itemize}
\end{remark}


\end{document}